\documentclass[10pt]{article}
\usepackage[affil-it]{authblk}
\usepackage{amsmath,amsthm,amssymb,cite}
\usepackage{lineno,hyperref}

\usepackage{calrsfs}

\usepackage{tikz}
\usetikzlibrary{arrows,chains,matrix,positioning,scopes}
\makeatletter
\tikzset{join/.code=\tikzset{after node path={%
\ifx\tikzchainprevious\pgfutil@empty\else(\tikzchainprevious)%
edge[every join]#1(\tikzchaincurrent)\fi}}}
\makeatother
\tikzset{>=stealth',every on chain/.append style={join},
         every join/.style={->}}

\newtheorem{theorem}{Theorem}[section]

\newtheorem{corollary}[theorem]{Corollary}
\theoremstyle{definition}
\newtheorem{remark}[theorem]{Remark}
\newtheorem{Definition}[theorem]{Definition}

\newtheorem{proposition}[theorem]{Proposition}

\title{\v{C}ech border homology and cohomology groups and some applications{\footnote{The author was supported by grant FR/233/5-103/14 from Shota Rustaveli National Science Foundation (SRNSF)}}}
\date{\vspace{-5ex}}

\author{Vladimer Baladze}
\affil{Department of Mathematics\\Batumi Shota Rustaveli  State University}
\providecommand{\keywords}[1]{\textbf{\textit{Keywords and Phrases:}} #1}
\begin{document}
\maketitle

\begin{abstract}
In the paper the \v{C}ech border homology and cohomology groups of closed pairs of normal spaces are constructed and investigated. These groups give intrinsic characterizations of \v{C}ech homology and cohomology groups based on finite open coverings, homological and cohomological coefficients of cyclicity, small and large cohomological dimensions of remainders of Stone-\v{C}ech compactifications of metrizable spaces. 
\end{abstract}

\section*{Introduction}
The investigation and discussion presented in this paper are centered around the following problem:  

Find necessary and sufficient conditions under which a space of given class has a compactification whose remainder has the given topological property (cf. [Sm$_{2}$], Problem I, p.332 and Problem II, p.334).

This problem for different topological invariants and properties was studied by several authors: 
\begin{itemize}
\item[$\bullet$] J.M.Aarts [A], J.M.Aarts and T.Nishiura [A-N], Y. Akaike, N. Chinen and  K. Tomoyasu [Ak-Chin-T], V.Baladze [B$_1$], M.G. Charalambous [Ch], A.Chigogidze ([Chi$_1$], [Chi$_2$]),  H. Freudenthal ([F$_1$],[F$_{2}$]), K.Morita [Mo], E.G. Skljarenko [Sk], Ju.M.Smirnov ([Sm$_1$]-[Sm$_5$]) and H.De Vries [V] found conditions under which the spaces have extensions whose remainders have given covering and inductive dimensions, and combinatorial properties. 
\item[$\bullet$] The remainders of finite order extensions are defined and investigated by H.Inasaridze ([I$_1$], [I$_{2}$]). Using the results obtained in these papers, H.Inasaridze [I$_3$], L.Zambakhidze ([Z$_{1}$],[Z$_2$]), and I.Tsereteli [Ts] solved interesting problems of homological algebra, general topology and dimension theory.
\item[$\bullet$] $n$-dimensional (co)homology groups of remainders of precompact spaces are studied by V.Baladze [B$_3$], V.Baladze and L.Turmanidze [B-Tu].
\item[$\bullet$] A.Calder [C] described $n$-dimensional cohomotopy groups of remainders of Stone-\v{C}ech compactifications.
\item[$\bullet$] The characterizations of shapes of remainders of spaces are established in papers of V.Baladze ([B$_2$],[B$_{3}$]), B.J.Ball [Ba], J.Keesling ([K$_{1}$], [K$_2$]), J.Keesling and R.B. Sher [K-Sh].
\end{itemize}

The present paper is motivated by the general problem mentioned above. Specifically, we study this problem for the properties: \v{C}ech (co)homology groups based on finite open covers, coefficients of cyclicity and cohomological dimensions of remainders of Stone-\v{C}ech compactifications of metrizable spaces are given groups and given numbers, respectively.

In this paper we define the \v{C}ech type covariant and contravariant functors which coefficients in an abelian group $G$, 
\[\check{\rm{H}}_n^{\infty}(-,-;G):\mathcal{N}^{\bf 2}_{p}\to \mathcal{A}b\]
and 
\[\hat{\rm{H}}^n_{\infty}(-,-;G):\mathcal{N}^{\bf 2}_{p}\to \mathcal{A}b,\]
from the category $\mathcal{N}^{2}_{p}$ of closed pairs of normal spaces and proper maps to the category $\mathcal{A}b$ of abelian groups and homomorphisms. The construction of these functors is based on all border open covers of pairs $(X,A)\in ob(\mathcal{N}^{\bf 2}_{p})$ (see Definition 1.1 and Definition 1.2).

One of our main results of the paper is the following theorem (see Theorem 2.1). Let $\mathcal{M}^{2}_{p}$ be the category of closed pairs of metrizable spaces and proper maps. For each closed pair $(X,A)\in ob(\mathcal{M}^{\bf 2}_{p})$, one has
\[\check{{H}}_n^f(\beta X\setminus X, \beta A\setminus A;G)=\check{{H}}_n^{\infty}(X,A;G)\]
and
\[\hat{{H}}^n_f(\beta X\setminus X, \beta A\setminus A;G)=\hat{{H}}^n_{\infty}(X,A;G),\]
where $\check{{H}}_n^f(\beta X\setminus X, \beta A\setminus A;G)$ and $\hat{{H}}^n_f(\beta X\setminus X, \beta A\setminus A;G)$ are \v{C}ech homology and cohomology groups based on all finite open covers of $(\beta X\setminus X, \beta A\setminus A)$, respectively (see [E-St], Ch. IX, p.237).

We also consider the border cohomological and homological coefficients of cyclicity $\eta_G^{\infty}$ and $\eta^G_{\infty}$, border small and large cohomological dimensions $d^f_{\infty}(X;G)$ and $D^f_{\infty}(X;G)$ and prove the following relations (see Theorem 2.3, Theorem 2.5 and Theorem 2.8):
\[\eta_G^{\infty}(X,A)=\eta_G(\beta X\setminus X, \beta A\setminus A),\]
\[\eta_{\infty}^{G}(X,A)=\eta^{G}(\beta X\setminus X,\beta A\setminus A),\] 
\[d^f_{\infty}(X;G) = d_f(\beta X\setminus X;G),\]
\[D^f_{\infty}(X;G) = D_f(\beta X\setminus X;G),\]
where $\eta_G(\beta X\setminus X, \beta A\setminus A)$, $\eta^{G}(\beta X\setminus X, \beta A\setminus A)$ and $d_f(\beta X\setminus X;G)$, $D_f(\beta X\setminus X;G)$ are well known cohomological coefficient of cyclicity [No], homological coefficient of cyclicity (see Definition 2.2) and small cohomological dimension, large cohomological dimension [N] of remainders $(\beta X\setminus X, \beta A\setminus A)$ and $\beta X\setminus X$, respectively.

Without any further reference we will use definitions and results from the monographs General Topology [En], Algebraic Topology [E-St] and Dimension Theory [N].

\section{On \v{C}ech border homology and cohomology groups}
In this section we give an outline of a generalization of \v{C}ech homology theory by replacing the set of all finite open coverings in the definition of \v{C}ech (co)homology group ($\hat{{H}}^n_f(X,A;G)$) $\check{{H}}_n^f(X,A;G)$  (see [E-St],Ch.IX, p.237) by a set of all finite open families with compact enclosures. For this aim we give the following definitions.

An indexed family of subsets of set $X$ is a function $\alpha$ from an indexed set $V_{\alpha}$ to the set $2^{X}$ of subsets of $X$. The image $\alpha(v)$ of index $v\in V_{\alpha}$ is denoted by $\alpha_{v}$. Thus the indexed family $\alpha$ is the family $\alpha=\{\alpha_{v}\}_{v\in V_{\alpha}}$. If $|V_{\alpha}|<\aleph_{0}$, then we say that $\alpha$ family is a finite family.

Let $V^{'}_{\alpha}$ be a subset of set $V_{\alpha}$. A family $\{\alpha_{v}\}_{v\in V_{\alpha}^{'}}$ is called a subfamily of family $\{\alpha_{v}\}_{v\in V_{\alpha}}$. 

By $\alpha=\{\alpha_{v}\}_{v\in (V_{\alpha}, V_{\alpha}^{'})}$ we denote the family consisting of family $\{\alpha_{v}\}_{v\in V_{\alpha}}$ and its subfamily $\{\alpha_{v}\}_{v\in V_{\alpha}^{'}}$.  

\begin{Definition}(see [Sm$_4$]).
A finite family $\alpha=\{\alpha_{v}\}_{v\in V_{\alpha}}$ of open subsets of normal space $X$ is called a border cover of $X$ if its enclosure $K_{\alpha}=X\setminus \bigcup\limits_{v\in {{V}_{\alpha }}}{{{\alpha }_{v}}}$ is a compact subset of $X$.
\end{Definition}

\begin{Definition}(cf. [Sm$_{4}$]). A finite open family $\alpha=\{\alpha_{v}\}_{v\in (V_{\alpha}, V_{\alpha}^{A})}$ is called a border cover of closed pair $(X,A)\in \mathcal{N}^{2}$ if there exists a compact subset $K_{\alpha}$ of $X$ such that $X\setminus K_{\alpha}=\bigcup\limits_{v\in {{V}_{\alpha }}}{{{\alpha }_{v}}}$ and $A\setminus K_{\alpha}\subseteq \bigcup\limits_{v\in {{V}_{\alpha }^{A}}}{{{\alpha }_{v}}}$.
\end{Definition}

The set of all border covers of $(X,A)$ is denoted by ${\rm cov}_{\infty}(X,A)$. Let $K_{\alpha}^{A}=K_{\alpha}\cap A$. Then the family $\{\alpha_{v}\cap A\}_{v\in V_{\alpha}^{A}}$ is a border cover of subspace $A$. 

\begin{Definition}
Let $\alpha, \beta \in {\rm cov}_{\infty}(X,A)$ be two border covers of $(X,A)$ with indexing pairs $(V_{\alpha},V_{\alpha}^{A})$ and $(V_{\beta},V_{\beta}^{A})$, respectively. We say that the border cover $\beta$ is a refinement of border cover $\alpha$ if there exists a refinement projection function $p:(V_{\beta},V_{\beta}^{A})\to (V_{\alpha},V_{\alpha}^{A})$ such that for each index $v\in V_{\beta}$ ($v\in V_{\beta}^{A}$) $\beta_{v}\subset \alpha_{p(v)}$.
\end{Definition}

It is clear that ${\rm cov}_{\infty}(X,A)$ becomes a directed set with the relation $\alpha \leq \beta$ whenever $\beta$ is a refinement of $\alpha$.

Note that for each $\alpha \in {\rm cov}_{\infty}(X,A)$, $\alpha \leq \alpha$, and if for each $\alpha, \beta, \gamma \in {\rm cov}_{\infty}(X,A)$, $\alpha\leq \beta$ and $\beta \leq \gamma$, then $\alpha \leq \gamma$.

Let $\alpha, \beta \in {\rm cov}_{\infty}(X,A)$ be two border covers with indexing pairs $(V_{\alpha},V_{\alpha}^{A})$ and $(V_{\beta},V_{\beta}^{A})$, respectively. Consider a family $\gamma=\{\gamma_{v}\}_{v\in (V_{\gamma},V_{\gamma}^{A})}$, where $V_{\gamma}=V_{\alpha}\times V_{\beta}$ and $V_{\gamma}^{A}=V_{\alpha}^{A}\times V_{\beta}^{A}$. Let $v=(v_{1},v_{2})$, where $v_{1}\in V_{\alpha}$, $v_{2}\in V_{\beta}$. Assume that $\gamma_{v}=\alpha_{v_{1}}\cap \beta_{v_{2}}$. The family $\gamma=\{\gamma_{v}\}_{v\in (V_{\gamma},V_{\gamma}^{A})}$ is a border cover of $(X,A)$ and $\gamma \geq \alpha,\beta$.

For each border cover $\alpha\in \rm{cov}^{\infty}(X,A)$ with indexing pair $(V_{\alpha},V_{\alpha}^{A})$, by $(X_{\alpha},A_{\alpha})$ denote the nerve $\alpha$, where $A_{\alpha}$ is the subcomplex of simplexes $s$ of complex $X_{\alpha}$ with vertices of $V_{\alpha}^{A}$ such that $\rm{Car}_{\alpha}(s)\cap A\neq \emptyset$, where ${\rm Car}_{\alpha}(s)$ is the carrier of simplex $s$ (see [E-St], pp.234). The pair $(X_{\alpha},A_{\alpha})$ is a simplicial pair. Moreover, any two refinement projection functions $p,q:\beta \to \alpha$ induce contiguous simplicial maps of simplicial pairs $p_{\alpha}^{\beta},q_{\alpha}^{\beta}:(X_{\beta},A_{\beta})\to (X_{\alpha},A_{\alpha})$ (see [E-St], pp. 234-235).

Using the construction of formal homology theory of simplicial complexes ([E-St], Ch.VI) we can define the unique homomorphisms
\[p_{\alpha {*}}^{\beta}:H_{n}(X_{\beta},A_{\beta}:G)\to H_{n}(X_{\alpha},A_{\alpha};G)\]
and
\[p_{\alpha}^{\beta {*}}:H^{n}(X_{\alpha},A_{\alpha}:G)\to H_{n}(X_{\beta},A_{\beta};G),\]
where $G$ is any abelian coefficient group. 

Note that $p_{\alpha {*}}^{\alpha}=1_{H_{n}(X_{\alpha},A_{\alpha}:G)}$ and $p_{\alpha}^{\alpha {*}}=1_{H^{n}(X_{\alpha},A_{\alpha}:G)}$. If $\gamma \geq \beta \geq \alpha$ than 
\[p_{\alpha {*}}^{\gamma}=p_{\alpha {*}}^{\beta}\cdot p_{\beta {*}}^{\gamma}\]
and
\[p_{\alpha}^{\gamma {*}}=p_{\beta}^{\gamma *}\cdot p_{\alpha}^{\beta *}.\]

Thus, the families
\[\{H_{n}(X_{\alpha},A_{\alpha};G),p_{\alpha *}^{\beta}, {\rm cov}_{\infty}(X,A)\}\]
and
\[\{H^{n}(X_{\alpha},A_{\alpha};G),p_{\alpha }^{\beta *}, {\rm cov}_{\infty}(X,A)\}\]
form inverse and direct systems of groups.

The inverse and direct limit groups of above defined inverse and direct systems are denoted by symbols
\[\check{H}_{n}^{\infty}(X,A;G)=\underset{\longleftarrow }{\mathop{\lim }}\,\{H_{n}(X_{\alpha},A_{\alpha};G),p_{\alpha *}^{\beta},  {\rm cov}_{\infty}(X,A)\}\]
and
\[\hat{H}^{n}_{\infty}(X,A;G)=\underset{\longrightarrow}{\mathop{\lim }}\,\{H^{n}(X_{\alpha},A_{\alpha};G),p_{\alpha }^{\beta *},  {\rm cov}_{\infty}(X,A)\}\]
and called $n$-dimensional \v{C}ech border homology group and $n$-dimensional \v{C}ech border cohomology group of pair $(X,A)$ with coefficients in abelian group $G$, respectively.

According to [E-St] a border cover $\alpha\in {\rm cov}_{\infty}(X,A)$ indexed by $(V_{\alpha},V_{\alpha}^{A})$ is called proper if $V_{\alpha}^{A}$ is the set of all $v\in V_{\alpha}$ with $\alpha_{v}\cap A\neq \emptyset$. The set of proper border covers is denoted by ${\rm{Pcov}_{\infty}}(X,A)$. Now define a function 
\[\rho:{\rm cov}_{\infty}(X)\to {\rm cov}_{\infty}(X,A)\]
By definition, for each border cover of $X$ $\alpha=\{\alpha_{v}\}_{v\in V_{\alpha}}$ $$\rho (\alpha)=\{\alpha_{v}\}_{v\in (V_{\alpha},V^{'})},$$ where $V^{'}$ is the set of $v\in V_{\alpha}$ for which $\alpha_{v}\cap A\neq \emptyset$. It is clear that the family $\rho (\alpha)$ is a proper border cover and the function $\rho :{\rm cov}_{\infty}(X)\to {\rm{Pcov}_{\infty}}(X,A)$ induced by $\rho$ is one to one. Moreover, if $\alpha^{'}\leq \alpha$, then $\rho(\alpha^{'})\leq \rho (\alpha)$.

\begin{proposition}\label{Pro. 1.4}
For each pair $(X,A)\in ob (\mathcal{N}^{2}_{p})$ the set ${\rm{Pcov}_{\infty}}(X,A)$ of proper border covers of $(X,A)$ is a cofinal subset of ${\rm cov}_{\infty}(X,A)$.
\end{proposition}
\begin{proof}
Let $\alpha=\{a_{v}\}_{v\in (V_{\alpha},V_{\alpha}^{A})}$ be a border cover of $(X,A)$. Assume that
\[V^{'}=\{\alpha_{v}|\alpha_{v}\cap A\neq \emptyset,v\in V_{\alpha}^{A})\}.\]
Consider a family $\beta=\{\beta_{v}\}_{v\in (V_{\alpha},V^{'})}$ consisting of subsets
\[\beta_{v}=\alpha_{v}\setminus A,~~~~~~v\in V_{\alpha}\setminus V^{'}\]
and
\[\beta_{v}=\alpha_{v},~~~~~v\in V^{'}.\]
Note that $\beta$ is a border cover of $(X,A)$ and $\beta\geq \alpha$.
\end{proof}

Consequently, in definitions of \v{C}ech border homology and cohomology groups of pairs $(X,A)\in ob(\mathcal{N}^{2}_{p})$ we may replace the set ${\rm cov}_{\infty}(X,A)$ by the subset ${\rm{Pcov}_{\infty}}(X,A)$.

Now we define, for a given proper map $f:(X,A)\to (Y,B)$ of pairs, the induced homomorphisms
\[f^{\infty}_{*}:\check{H}_{n}^{\infty}(X,A;G)\to \check{H}_{n}^{\infty}(Y,B;G)\]
and
\[f^{*}_{\infty}:\hat{H}^{n}_{\infty}(X,A;G)\to \hat{H}^{n}_{\infty}(Y,B;G).\]

Let $\alpha\in {\rm cov}_{\infty}(Y,B)$ be a border cover with index set $V_{\alpha}$ and $K_{\alpha}=Y\setminus \bigcup\limits_{v\in {{V}_{\alpha }}}{{{\alpha }_{v}}}$. Consider a family $\alpha^{'}=\{f^{-1}(\alpha_{v})\}_{v\in V_{\alpha}}$. Note that 

\[X\setminus \bigcup\limits_{v\in {{V}_{\alpha }}}{{{f^{-1}(\alpha }_{v})}}=X\setminus f^{-1}(\bigcup\limits_{v\in {{V}_{\alpha }}}{{{\alpha }_{v}}})=X\setminus f^{-1}(Y\setminus K_{\alpha})=f^{-1}(K_{\alpha}).
\]

Let $\alpha^{'}_{v}=f^{-1}(\alpha_{v})$ and $V_{\alpha^{'}}=V_{\alpha}$. Since $f$ is proper, $f^{-1}(K_{\alpha})$ is a compact subset of $X$.

Since $B\setminus K_{\alpha}\subseteq \bigcup\limits_{v\in {{V}_{\alpha }^{B}}}{{{\alpha }_{v}}}$, the subfamily $\{f^{-1}(\alpha_{v})|v\in V_{\alpha}^{B}\}$ is such that $A\setminus f^{-1}(K_{\alpha})\subseteq \bigcup\limits_{v\in {{V}_{\alpha }^{B}}}{{{f^{-1}(\alpha }_{v})}}$. Let $V_{\alpha^{'}}^{A}=V_{\alpha}^{B}$ and $K_{\alpha^{'}}=f^{-1}(K_{\alpha})$. Note that 
$A\setminus K_{\alpha^{'}}\subset \bigcup\limits_{v\in {{V}_{\alpha^{'}}^{A}}}{{{f^{-1}(\alpha }_{v})}}.$
Hence, $\alpha^{'}=\{f^{-1}(\alpha_{v})\}_{v\in (V_{\alpha^{'}},V_{\alpha^{'}}^{A})}$ is a border cover of pair $(X,A)$.

It is clear that $X_{\alpha^{'}}$ is a subcomplex of $Y_{\alpha}$ and $A_{\alpha^{'}}$ is a subcomplex of $B_{\alpha}$. By a symbol $f_{\alpha}:(X_{\alpha^{'}},A_{\alpha^{'}})\to (Y_{\alpha},B_{\alpha})$ denote the simplicial inclusion of $(X_{\alpha^{'}},A_{\alpha^{'}})$ into $(Y_{\alpha},B_{\alpha})$.

If $\alpha,\beta \in {\rm cov}_{\infty}(Y,B)$ and $\beta \geq \alpha$, then the diagrams
\begin{center}
\begin{tikzpicture}
\node (A) {$H_{n}(X_{\beta^{'}},A_{\beta^{'}};G)$};
\node (B) [node distance=4cm, right of=A] {$H_{n}(X_{\beta},A_{\beta};G)$};
\node (C) [node distance=2cm, below of=A] {$H_{n}(X_{\alpha^{'}},A_{\alpha^{'}};G)$};
\node (D) [node distance=4cm, right of=C] {$H_{n}(X_{\alpha},A_{\alpha};G)$};
\draw[->] (A) to node [above]{$f_{\beta *}$} (B);
\draw[->] (A) to node [left]{$p_{\alpha^{'}*}^{\beta^{'}}$}(C);
\draw[->] (C) to node [below]{$f_{\alpha *}$} (D);
\draw[->] (B) to node [right]{$p_{\alpha *}^{\beta}$} (D);
\end{tikzpicture}
\end{center}
and
\begin{center}
\begin{tikzpicture}
\node (A) {$H^{n}(X_{\alpha},A_{\alpha};G)$};
\node (B) [node distance=4cm, right of=A] {$H^{n}(X_{\alpha^{'}},A_{\alpha^{'}};G)$};
\node (C) [node distance=2cm, below of=A] {$H^{n}(X_{\beta},A_{\beta};G)$};
\node (D) [node distance=4cm, right of=C] {$H^{n}(X_{\beta^{'}},A_{\beta^{'}};G).$};
\draw[->] (A) to node [above]{$f_{\alpha}^{*}$} (B);
\draw[->] (A) to node [left]{$p_{\alpha}^{\beta *}$}(C);
\draw[->] (C) to node [below]{$f_{\beta}^{*}$} (D);
\draw[->] (B) to node [right]{$p_{\alpha^{'}}^{\beta^{'}*}$} (D);
\end{tikzpicture}
\end{center}
commute.

Thus, for each $\alpha \in {\rm cov}_{\infty}(Y,B)$, the induced homomorphisms $f_{\alpha *}$ and $f_{\alpha}^{*}$ together with function $\varphi :{\rm cov}_{\infty}(Y,B)\to {\rm cov}_{\infty}(X,A)$ given by formula $$\varphi(\alpha)=f^{-1}(\alpha), \alpha \in {\rm cov}_{\infty}(Y,B)$$ form maps 
\[(f_{\alpha *},\varphi):\{H_{n}(X_{\alpha^{'}},A_{\alpha^{'}}),p_{\alpha^{'} *}^{\beta^{'}},{\rm cov}_{\infty}(X,A)\}\to \{H_{n}(Y_{\alpha},A_{\alpha}),p_{\alpha *}^{\beta},{\rm cov}_{\infty}(Y,B)\}\]
and
\[(f_{\alpha}^{*},\varphi):\{H^{n}(Y_{\alpha},A_{\alpha}),p_{\alpha}^{\beta *},{\rm cov}_{\infty}(Y,B)\} \to \{H^{n}(X_{\alpha^{'}},A_{\alpha^{'}}),p_{\alpha^{'}}^{\beta^{'}_{*}},{\rm cov}_{\infty}(X,A)\}.\]

The limits of maps $(f_{\alpha *},\varphi)$ and $(f_{\alpha}^{*},\varphi)$ are denoted by \[f^{\infty}_{*}:\check{H}_{n}^{\infty}(X,A;G)\to \check{H}_{n}^{\infty}(Y,B;G)\]
and
\[f^{*}_{\infty}:\hat{H}^{n}_{\infty}(Y,B;G)\to \hat{H}^{n}_{\infty}(X,A;G)\]
and called homomorphisms induced by proper map $f:(X,A)\to (Y,B)$.

Note that if $f:(X,A)\to (Y,B)$ is the identity map, then the induced homomorphisms $f_{*}^{\infty}:\check{H}_{n}^{\infty}(X,A;G)\to \check{H}_{n}^{\infty}(Y,B;G)$ and $f^{*}_{\infty}:\hat{H}^{n}_{\infty}(Y,B;G)\to \hat{H}^{n}_{\infty}(X,A;G)$ are the identity homomorphisms. Furthermore, for each proper maps $f:(X,A)\to (Y,B)$ and $g:(Y,B)\to (Z,C)$
\[(g\cdot f)_{*}^{\infty}=g_{*}^{\infty}\cdot f_{*}^{\infty}\]
and
\[(g\cdot f)^{*}_{\infty}=f^{*}_{\infty}\cdot g^{*}_{\infty}.\]

We have the following theorem.
\begin{theorem}\label{theorem 1.4}
There exist the covariant and contravariant functors
\[\check{\rm{H}}_{*}^{\infty}(-,-;G):\mathcal{N}^{2}_{p}\to \mathcal{A}b\]
and
\[\hat{\rm{H}}^{*}_{\infty}(-,-;G):\mathcal{N}^{2}_{p}\to \mathcal{A}b\]
given by formulas
\[\check{\rm{H}}_{*}^{\infty}(-,-;G)(X,A)=\check{{H}}_{*}^{\infty}(X,A;G),~~~(X,A)\in ob(\mathcal{N}^{2}_{p})\]
\[\check{\rm{H}}_{*}^{\infty}(-,-;G)(f)=f_{*}^{\infty},~~~f\in \rm{Mor}_{\mathcal{N}^{2}_{p}}((X,A),(Y,B))\]
and
\[\hat{\rm{H}}^{*}_{\infty}(-,-;G)(X,A)=\hat{{H}}^{*}_{\infty}(X,A;G),~~~(X,A)\in ob(\mathcal{N}^{2}_{p})\]
\[\hat{\rm{H}}^{*}_{\infty}(-,-;G)(f)=f^{*}_{\infty},~~~f\in \rm{Mor}_{\mathcal{N}^{2}_{p}}((X,A),(Y,B)).\]
\end{theorem}
\begin{proof}
The proof follows from above discussion.
\end{proof}
We will call the functors $\check{\rm{H}}_{*}^{\infty}(-,-;G)$ and $\hat{\rm{H}}^{*}_{\infty}(-,-;G)$  \v{C}ech border homology and cohomology functors, respectively.

Now we define boundary and coboundary homomorphisms 
\[\partial_{n}^{\infty}:\check{{H}}_{n}^{\infty}(X,A;G)\to \check{{H}}_{n-1}^{\infty}(A;G)\]
and
\[\delta^{n}_{\infty}:\hat{{H}}^{n-1}_{\infty}(A;G)\to \hat{{H}}^{n}_{\infty}(X,A;G).\]

Let $(X,A)\in ob(\mathcal{N}^{2}_{p})$, $\beta,\alpha\in {\rm cov}_{\infty}(X,A)$ and $\beta \geq \alpha$. The refinement projection functions induce the unique homomorphisms $p_{\alpha *}^{\beta}:H_{n}(A_{\beta};G)\to H_{n}(A_{\alpha};G)$ and $p_{\alpha}^{\beta *}:H^{n}(A_{\alpha};G)\to H^{n}(A_{\beta};G)$, $p_{\alpha *}^{\beta}:H_{n}(X_{\beta};G)\to H_{n}(X_{\alpha};G)$ and $p_{\alpha}^{\beta *}:H_{n}(X_{\alpha};G)\to H_{n}(X_{\beta};G)$, which form inverse systems
\begin{center}
$\{H_{n}(A_{\alpha};G),p_{\alpha *}^{\beta}, {\rm cov}_{\infty}(X,A)\}$ and $\{H_{n}(X_{\alpha};G),p_{\alpha *}^{\beta},{\rm cov}_{\infty}(X,A)\}$
\end{center}
and direct systems 
\begin{center}
$\{H^{n}(A_{\alpha};G),p_{\alpha}^{\beta *}, {\rm cov}_{\infty}(X,A)\}$ and $\{H^{n}(X_{\alpha};G),p_{\alpha}^{\beta *}, {\rm cov}_{\infty}(X,A)\}.$
\end{center}

Let
\[\check{H}_{n}^{\infty}(A;G)_{(X,A)}=\underset{\longleftarrow }{\mathop{\lim }}\,\{H_{n}(A_{\alpha};G),p_{\alpha *}^{\beta},{\rm cov}_{\infty}(X,A)\},\]
\[\check{H}_{n}^{\infty}(X;G)_{(X,A)}=\underset{\longleftarrow }{\mathop{\lim }}\,\{H_{n}(X_{\alpha};G),p_{\alpha *}^{\beta},{\rm cov}_{\infty}(X,A)\},\]
\[\hat{H}^{n}_{\infty}(A;G)^{(X,A)}=\underset{\longrightarrow }{\mathop{\lim }}\,\{H^{n}(A_{\alpha};G),p_{\alpha}^{\beta *},{\rm cov}_{\infty}(X,A)\},\]
\[\hat{H}^{n}_{\infty}(X;G)^{(X,A)}=\underset{\longrightarrow }{\mathop{\lim }}\,\{H^{n}(X_{\alpha};G),p_{\alpha}^{\beta *},{\rm cov}_{\infty}(X,A)\}.\]

Our main aim is to show that the groups $\check{{H}}_{n}^{\infty}(A;G)$ and $\hat{H}_{n}^{\infty}(A;G)_{(X,A)}$, $\hat{{H}}^{n}_{\infty}(A;G)$ and $\hat{H}^{n}_{\infty}(A;G)^{(X,A)}$, $\check{H}_{n}(X;G)$ and $\check{H}_{n}(X;G)_{(X,A)}$, $\hat{H}^{n}(X;G)$ and $\hat{H}^{n}(X;G)^{(X,A)}$ are isomorphical groups.

Next we define a function $\varphi:{\rm cov}_{\infty}(X,A)\to {\rm cov}_{\infty}(A,\emptyset)$. Let $\alpha=\{\alpha_{v}\}_{v\in (V_{\alpha},V_{\alpha}^{A})}\in {\rm cov}_{\infty}(X,A)$. Assume that $(\varphi(\alpha))_{v}=\alpha_{v}\cap A$ for $v\in V_{\alpha}^{A}$. We have defined the border cover $\varphi(\alpha)\in {\rm cov}_{\infty}(A,\emptyset)$ indexed by pair $(V_{\alpha},\emptyset)$. 

Let $K_{\alpha}=X\setminus \bigcup\limits_{v\in {{V}_{\alpha }}}{{{\alpha }_{v}}}$. Note that
\[A\setminus (K_{\alpha}\cap A)=\bigcup\limits_{v\in {{V}_{\alpha }^{A}}}({{{\alpha }_{v}\cap A}})=\bigcup\limits_{v\in {{V}_{\alpha }^{A}}}{{{(\varphi(\alpha))_{v} }}}.\]

It is clear that $K_{\alpha}\cap A$ is a compact subset of the subspace $A$. Thus, $\varphi(\alpha)\in {\rm cov}_{\infty}(A,\emptyset)$. The defined function is an order preserving function.

It is easy to show that the image of function $\varphi$ is a cofinal subset of set ${\rm cov}_{\infty}(A,\emptyset)$. Note that $A_{\alpha}=A_{\varphi(\alpha)}$. By $\varphi_{\alpha}:A_{\varphi(\alpha)}\to A_{\alpha}$ denote this simplicial isomorphism. Hence, the family of pairs $(\varphi_{\alpha},\varphi)$ induces a map of inverse systems and direct systems
\[(\varphi_{\alpha *},\varphi):\{H_{n}(A_{\alpha};G), p_{\alpha *}^{\beta},{\rm cov}_{\infty}(A,\emptyset)\} \to \{H_{n}(A_{\alpha};G), p_{\alpha *}^{\beta},{\rm cov}_{\infty}(X,A)\}\]
and 
\[(\varphi_{\alpha}^{*},\varphi):\{H^{n}(A_{\alpha};G), p_{\alpha}^{\beta*},{\rm cov}_{\infty}(X,A)\} \to \{H_{n}(A_{\alpha};G), p_{\alpha}^{\beta*},{\rm cov}_{\infty}(A,\emptyset)\}.\]

Let $\Phi_{n}=\underset{\longleftarrow }{\mathop{\lim }}\,(\varphi_{\alpha *},\varphi)$ and $\Phi^{n}=\underset{\longrightarrow }{\mathop{\lim }}\,(\varphi_{\alpha}^{*},\varphi)$. Since all homomorphisms $\varphi_{\alpha *}$ and $\varphi_{\alpha}^{*}$ are isomorphisms, the limit homomorphisms
\[\Phi_{n}:\check{{H}}_{n}^{\infty}(A;G)\to \check{H}_{n}^{\infty}(A;G)_{(X,A)}\]
and 
\[\Phi^{n}:\hat{{H}}^{n}_{\infty}(A;G)^{(X,A)}\to \hat{{H}}^{n}_{\infty}(A;G)\]
are isomorphisms. 

Lets us also define a function $\psi :{\rm cov}_{\infty}(X,A)\to {\rm cov}_{\infty}(X,\emptyset)$. For each $\alpha=\{\alpha_{v}\}_{v\in (V_{\alpha},V_{\alpha}^{A})}\in {\rm cov}_{\infty}(X,A)$ assume that $(\psi(\alpha))_{v}=\alpha_{v}$, $v\in V_{\alpha}$. The family $\psi(\alpha)$ is indexed by $(V_{\alpha},\emptyset)$ and $\psi(\alpha)\in {\rm cov}_{\infty}(X,\emptyset)$.

Note that $X_{\alpha}=X_{\psi(\alpha)}$. Let $\psi_{\alpha}:X_{\psi(\alpha)}\to X_{\alpha}$ be a simplicial isomorphism. The family of pairs $(\psi_{\alpha},\psi)$ induce the maps of inverse and direct systems 
\[(\psi_{\alpha *},\psi):\{H_{n}(X_{\alpha};G),p_{\alpha *}^{\beta},{\rm cov}_{\infty}(X,\emptyset)\}\to \{H_{n}(X_{\alpha};G),p_{\alpha *}^{\beta},{\rm cov}_{\infty}(X,A)\}\]
and
\[(\psi_{\alpha}^{*},\psi):\{H^{n}(X_{\alpha};G),p_{\alpha}^{\beta *},{\rm cov}_{\infty}(X,A)\}\to \{H^{n}(X_{\alpha};G),p_{\alpha}^{\beta *},{\rm cov}_{\infty}(X,\emptyset)\}.\]

Let $\Psi_{n}=\underset{\longleftarrow }{\mathop{\lim }}\,(\psi_{\alpha *},\psi)$ and $\Psi^{n}=\underset{\longrightarrow }{\mathop{\lim }}\,(\psi_{\alpha}^{*},\psi)$.
Since each $\psi_{\alpha *}$ and $\psi_{\alpha}^{*}$ are isomorphisms, the induced limit homomorphisms
\[\Psi_{n}:\check{{H}}_{n}^{\infty}(X;G)\to \check{H}_{n}^{\infty}(X;G)_{(X,A)}\]
and 
\[\Psi^{n}:\hat{{H}}^{n}_{\infty}(X;G)^{(X,A)}\to \hat{H}^{n}_{\infty}(X;G)\]
are isomorphisms. 

There exist the limit sequences

\begin{center}
\begin{tikzpicture}
\node (A) {$\cdots$};
\node (B) [node distance=2cm, right of=A] {$\check{H}_{n}^{\infty}(X,A;G)$};
\node (C) [node distance=3cm, right of=B] {$\check{H}_{n}^{\infty}(X;G)_{(X,A)}$};
\node (D) [node distance=3cm, right of=C] {$\check{H}_{n}^{\infty}(A;G)_{(X,A)}$};
\node (E) [node distance=3cm, right of=D] {$\check{H}_{n+1}^{\infty}(X,A;G)$};
\node (F) [node distance=2cm, right of=E] {$\cdots$};
\draw[->] (B) to node [above]{}(A);
\draw[->] (C) to node [above]{$j_{n}^{'\infty}$} (B);
\draw[->] (D) to node [above]{$i_{n}^{'\infty}$}(C);
\draw[->] (E) to node [above]{$\partial^{'\infty}_{n+1}$}(D);
\draw[->] (F) to node [above]{}(E);
\end{tikzpicture}
\end{center}
and 
\begin{center}
\begin{tikzpicture}
\node (A) {$\cdots$};
\node (B) [node distance=2cm, right of=A] {$\hat{{H}}^{n}_{\infty}(X,A;G)$};
\node (C) [node distance=3cm, right of=B] {$\hat{{H}}^{n}_{\infty}(X;G)^{(X,A)}$};
\node (D) [node distance=3cm, right of=C] {$\hat{{H}}^{n}_{\infty}(A;G)^{(X,A)}$};
\node (E) [node distance=3cm, right of=D] {$\hat{{H}}^{n+1}_{\infty}(X,A;G)$};
\node (F) [node distance=2cm, right of=E] {$\cdots$};
\draw[->] (A) to node [above]{}(B);
\draw[->] (B) to node [above]{$j^{'n}_{\infty}$} (C);
\draw[->] (C) to node [above]{$i^{'n}_{\infty}$}(D);
\draw[->] (D) to node [above]{$\delta_{\infty}^{'n}$}(E);
\draw[->] (E) to node [above]{}(F);
\end{tikzpicture}
\end{center}
generated by the families consisting of homology and cohomology sequences of simplicial pairs $(X_{\alpha},A_{\alpha}), \alpha \in {\rm cov}_{\infty}(X,A)$, respectively.

Consider the diagrams
\begin{center}
\begin{tikzpicture}
\node (A) {$\check{H}_{n}^{\infty}(X,A;G)$};
\node (B) [node distance=4cm, right of=A] {$\check{H}_{n-1}^{\infty}(A;G)_{(X,A)}$};
\node (C) [node distance=4cm, right of=B] {$\check{H}_{n-1}^{\infty}(A;G)$};
\draw[->] (A) to node [above]{$\partial^{'\infty}_{n}$} (B);
\draw[->] (C) to node [above]{$\Phi_{n-1}$}(B);
\end{tikzpicture}
\end{center}
and
\begin{center}
\begin{tikzpicture}
\node (A) {$\hat{H}^{n}_{\infty}(A;G)$};
\node (B) [node distance=4cm, right of=A] {$\hat{H}^{n}_{\infty}(A;G)^{(X,A)}$};
\node (C) [node distance=4cm, right of=B] {$\hat{H}_{\infty}^{n+1}(X,A;G)$};
\draw[->] (B) to node [above]{$\Psi^{n}$} (A);
\draw[->] (B) to node [above]{$\delta^{'n}_{\infty}$}(C);
\end{tikzpicture}
\end{center}
and define the boundary homomorphism of \v{C}ech border homology groups and coboundary homomorphism of \v{C}ech border cohomology groups as compositions 
\[\partial_{n}^{\infty}=(\Phi_{n-1})^{-1}\cdot \partial_{n}^{'\infty}\]
and
\[\delta_{\infty}^{n}=\delta^{'n}_{\infty}\cdot (\Psi^{n})^{-1}.\]

In this way we arrive to the following theorems.
\begin{theorem}\label{theorem 1.5}
Let $f:(X,A)\to (Y,B)$ be a proper map. Then hold the following equalities
\[(f_{|A})_{*}^{\infty}\cdot \partial_{n}^{\infty}=\partial_{n}^{\infty}\cdot f_{*}^{\infty}\]
and
\[\delta_{\infty}^{n-1}(f_{|A})^{*}_{\infty}=f^{*}_{\infty}\cdot \delta_{\infty}^{n-1}.\]
\end{theorem}
\begin{proof}

The desired equalities follow from the commutativity of the diagrams.
\begin{center}
\begin{tikzpicture}
\node (A) {$\check{H}_{n}^{\infty}(X,A;G)$};
\node (B) [node distance=4cm, right of=A] {$\check{H}_{n-1}^{\infty}(A;G)_{(X,A)}$};
\node (C) [node distance=4cm, right of=B] {$\check{H}_{n-1}^{\infty}(A;G)$};
\node (D) [node distance=2cm, below of=A] {$\check{H}_{n}^{\infty}(Y,B;G)$};
\node (E) [node distance=4cm, right of=D] {$\check{H}_{n-1}^{\infty}(B;G)_{(Y,B)}$};
\node (F) [node distance=4cm, right of=E] {$\check{H}_{n-1}^{\infty}(B;G)$};
\draw[->] (A) to node [above]{$\partial_{n}^{'\infty}$} (B);
\draw[->] (C) to node [above]{$\Phi_{n-1}$}(B);
\draw[->] (D) to node [below]{$\partial_{n}^{'\infty}$} (E);
\draw[->] (F) to node [below]{$\Phi_{n-1}$}(E);
\draw[->] (A) to node [left]{$f_{*}^{\infty}$} (D);
\draw[->] (B) to node [left]{$(f_{|A})_{*}^{'\infty}$} (E);
\draw[->] (C) to node [right]{$(f_{|A})_{*}^{\infty}$} (F);
\end{tikzpicture}
\end{center}
and
\begin{center}
\begin{tikzpicture}
\node (A) {$\hat{H}^{n-1}_{\infty}(B;G)$};
\node (B) [node distance=4cm, right of=A] {$\hat{H}^{n-1}_{\infty}(B;G)^{(Y,B)}$};
\node (C) [node distance=4cm, right of=B] {$\hat{H}^{n}_{\infty}(Y,B;G)$};
\node (D) [node distance=2cm, below of=A] {$\hat{H}^{n-1}_{\infty}(A;G)$};
\node (E) [node distance=4cm, right of=D] {$\hat{H}^{n-1}_{\infty}(A;G)^{(X,A)}$};
\node (F) [node distance=4cm, right of=E] {$\hat{H}^{n}_{\infty}(X,A;G),$};
\draw[->] (B) to node [above]{$\Phi^{n-1}$} (A);
\draw[->] (B) to node [above]{$\delta^{'n}_{\infty}$}(C);
\draw[->] (E) to node [below]{$\Phi^{n-1}$} (D);
\draw[->] (E) to node [below]{$\delta^{'n}_{\infty}$}(F);
\draw[->] (A) to node [left]{$(f_{|A})^{*}_{\infty}$} (D);
\draw[->] (B) to node [left]{$(f_{|A})^{'*}_{\infty}$} (E);
\draw[->] (C) to node [right]{$f^{*}_{\infty}$} (F);
\end{tikzpicture}
\end{center}
where $(f_{|A})_{*}^{'\infty}$ and $(f_{|A})^{'*}_{\infty}$ are defined as the appropriate limit homomorphisms.
\end{proof}

Let $i:A\to X$ and $j:X\to (X,A)$ be the inclusion maps.
\begin{theorem}\label{1.6}
Let $(X,A)\in ob(\mathcal{N}^{2}_{p})$. Then the \v{C}ech border cohomology sequence 
\begin{center}
\begin{tikzpicture}
\node (A) {$\cdots$};
\node (B) [node distance=2cm, right of=A] {$\check{H}^{n-1}_{\infty}(A;G)$};
\node (C) [node distance=3.5cm, right of=B] {$\check{H}^{n}_{\infty}(X,A;G)$};
\node (D) [node distance=3cm, right of=C] {$\check{H}^{n}_{\infty}(X;G)$};
\node (E) [node distance=2.5cm, right of=D] {$\check{H}^{n}_{\infty}(A;G)$};
\node (F) [node distance=2cm, right of=E] {$\cdots$};
\draw[->] (A) to node [above]{}(B);
\draw[->] (B) to node [above]{$\delta_{\infty}^{n-1}$} (C);
\draw[->] (C) to node [above]{$j^{*}_{\infty}$}(D);
\draw[->] (D) to node [above]{$i^{*}_{\infty}$}(E);
\draw[->] (E) to node [above]{}(F);
\end{tikzpicture}
\end{center}
is exact while the \v{C}ech border homology sequence
\begin{center}
\begin{tikzpicture}
\node (A) {$\cdots$};
\node (B) [node distance=2cm, right of=A] {$\hat{H}_{n-1}^{\infty}(A;G)$};
\node (C) [node distance=3.5cm, right of=B] {$\hat{H}_{n}^{\infty}(X,A;G)$};
\node (D) [node distance=3cm, right of=C] {$\hat{H}_{n}^{\infty}(X;G)$};
\node (E) [node distance=2.5cm, right of=D] {$\hat{H}_{n}^{\infty}(A;G)$};
\node (F) [node distance=2cm, right of=E] {$\cdots$};
\draw[->] (B) to node [above]{}(A);
\draw[->] (C) to node [above]{$\partial^{\infty}_{n}$} (B);
\draw[->] (D) to node [above]{$j_{*}^{\infty}$}(C);
\draw[->] (E) to node [above]{$i_{*}^{\infty}$}(D);
\draw[->] (F) to node [above]{}(E);
\end{tikzpicture}
\end{center}
is partially exact.
\end{theorem}
\begin{proof}
One can prove this theorem  analogously to the corresponding theorem of the classical \v{C}ech theory [E-St].
\end{proof}
%-------------
\begin{theorem}
Let $(X,A)\in ob(\mathcal{N}^{2}_{p})$ and $G$ be an abelian group. If $U$ is open in $X$ and $\bar{U}\subset {\rm int} A$, then the inclusion map $i:(X\setminus U,A\setminus U)\to (X,A)$ induces isomorphisms
\[i_{*}^{\infty}:\check{H}^{\infty}_{n}(X\setminus U,A\setminus U)\to \check{H}_{n}^{\infty}(X,A;G)\]
and
\[j^{*}_{\infty}:\hat{H}^{n}_{\infty}(X,A;G)\to \hat{H}_{\infty}^{n}(X\setminus U,A\setminus U)\]
\end{theorem}
\begin{proof}
Let ${\rm cov}_{\infty}^{'}(X,A)$ be the subset of ${\rm cov}(X,A)$ consisting of all covers $\alpha =\{\alpha_{v}\}_{v\in (V_{\alpha},V_{\alpha}^{A})}$ with property: 

if $\alpha_{v}\cap U\neq \emptyset$, then $v\in V_{\alpha}^{A}$ and $\alpha_{v}\subset A$. 

First we prove that ${\rm cov}_{\infty}^{'}(X,A)$ is cofinal in ${\rm cov}_{\infty}(X,A)$. Let $\alpha=\{\alpha_{v}\}_{v\in (V_{\alpha},V_{\alpha}^{A})}$ be a border cover of $(X,A)$ with enclosure $K_{\alpha}$. Let $V^{'}$ be a set such that $V^{'}\cap V_{\alpha}=\emptyset$ and there exists a bijective function between $V_{\alpha}^{A}$ and $V^{'}$. Let $v\in V_{\alpha}^{A}$. The correspondence element of $v$ in $V^{'}$ denote by $v^{'}$. Define the border cover $\gamma=\{\gamma_{v}\}_{v\in (V_{\alpha}\cup V^{'}, V_{\alpha}^{A}\cup V^{'})} \in {\rm cov}_{\infty}(X,A)$. Let 
\[\gamma_{v}=\alpha_{v}\setminus \bar{U},~~~~v\in V_{\alpha}\]
and
\[\gamma_{v^{'}}=\alpha_{v}\cap {\rm int} A,~~~~v^{'}\in V^{'}.\]

It is clear that $\gamma$ is a border cover of $(X,A)$ with enclosure $K_{\alpha}$ and $\gamma \geq \alpha$.

We prove that $i^{-1}({\rm cov}_{\infty}^{'}(X,A))$ is cofinal in ${\rm cov}_{\infty}(X\setminus U,A\setminus U)$. Let $\beta=\{\beta_{v}\}_{v\in (V_{\beta},V_{\beta}^{A\setminus U})}$ be a border cover of $(X\setminus U,A\setminus U)$ with enclosure $K_{\beta}$. Define a border cover $\alpha=\{\alpha_{v}\}_{v\in (V_{\beta},V_{\beta}^{A\setminus U})} \in {\rm cov}_{\infty}(X,A)$.

Let 
\[\alpha_{v}=\beta_{v}\cup U.\]

The family $\alpha=\{\alpha_{v}\}_{v\in (V_{\beta},V_{\beta}^{A\setminus U})}$ is a border cover of $(X,A)$ with enclosure $K_{\beta}$.

Let $\gamma\in {\rm cov}_{\infty}^{'}(X,A)$ be a border cover such that $\gamma \geq \alpha$. It is clear that $i^{-1}(\gamma)\geq \beta =i^{-1}(\alpha)$.

Note that
\[X_{\alpha}=(X\setminus U)_{\beta}\cup A_{\alpha}\]
and
\[(A\setminus U)_{\beta}=(X\setminus U)_{\beta}\cap A_{\alpha}.\]

As in [E-St] we can prove that there exist isomorphisms
\[i_{\alpha *}:H_{n}((X\setminus U)_{\beta},(A\setminus U)_{\beta};G)\to H_{n}(X_{\alpha},A_{\alpha};G)\]
and
\[i_{\alpha}^{*}:H^{n}(X_{\alpha},A_{\alpha};G)\to H^{n}((X\setminus U)_{\beta},(A\setminus U)_{\beta};G).\]

The conclusion of the theorem is a consequence of these isomorphisms.
\end{proof}
\begin{theorem}\label{theorem 1.7}
If $X$ is a compact space, then for each $n\neq 0$, 
\[\check{{H}}_{n}^{\infty}(X;G)=0=\check{{H}}_{\infty}^{n}(X;G)\]
and
\[\hat{{H}}_{0}^{\infty}(X;G)=G=\hat{{H}}_{\infty}^{0}(X;G).\]
\end{theorem}
\begin{proof}
Let $\alpha\in{\rm cov}_{\infty}(X)$ be the border cover of $X$ consisting of empty set. It is clear that $\alpha$ is a refinement of any border cover of $X$. The set $\{\alpha\}$ is a cofinal subset of ${\rm cov}_{\infty}(X)$. Consider the inverse system $\{H_{n}(X_{\alpha};G),p_{\alpha *}^{\alpha},\{\alpha\}\}$ and direct system $\{H^{n}(X_{\alpha};G),p^{\alpha *}_{\alpha},\{\alpha\}\}$. We have $$\underset{\longleftarrow}{\mathop{\lim }}\,\{{{H}_{n}}({{X}_{\alpha }};G),p_{\alpha *}^{\alpha },\{\alpha \}\}=\check{{H}}_{n}^{\infty}(X;G)=H_{n}(X_{\alpha};G)$$ and $$\underset{\longrightarrow}{\mathop{\lim }}\,\{{{H}^{n}}({{X}_{\alpha }};G),p^{\alpha *}_{\alpha },\{\alpha \}\}=\hat{{H}}^{n}_{\infty}(X;G)=H^{n}(X_{\alpha};G).$$ The nerve $X_{\alpha}$ consists of one vertex. Using the methods of proofs of results VI.3.8 and VI.4.3 of [E-St] we than conclude that
\[\check{{H}}_{n}^{\infty}(X;G)=0=\hat{{H}}_{\infty}^{n}(X;G)\]
and
\[\hat{{H}}_{0}^{\infty}(X;G)=G=\hat{{H}}_{\infty}^{0}(X;G).\]
\end{proof}

Thus, \v{C}ech border homology (cohomology) functors $\check{H}^{\infty}_{n}(-,-;G)(\hat{H}_{\infty}^{n}(-,-;G)):\mathcal{N}^{2}_{p}\to \mathcal{A}b$ satisfy the Steenrod-Eilenberg type axioms (cf.[E-St]): axiom of natural transformation, axiom of partially exactness (axiom of exactness), axiom of excision and axiom of dimension; but they do not satisfy the proper homotopy axiom.

The above obtained results yield the next theorem. 

\begin{theorem}
Let $(X,A,B)$ be a triple of normal space $X$ and its closed subsets $A$ and $B$ with $B\subset A$. Then the \v{C}ech border homology sequence 
\begin{center}
\begin{tikzpicture}
\node (A) {$\cdots$};
\node (B) [node distance=2cm, right of=A] {$\check{H}_{n-1}^{\infty}(A,B;G)$};
\node (C) [node distance=3cm, right of=B] {$\check{H}_{n}^{\infty}(X,A;G)$};
\node (D) [node distance=3cm, right of=C] {$\check{H}_{n}^{\infty}(X,B;G)$};
\node (E) [node distance=3cm, right of=D] {$\check{H}_{n}^{\infty}(A,B;G)$};
\node (F) [node distance=2cm, right of=E] {$\cdots$};
\draw[->] (B) to node [above]{}(A);
\draw[->] (C) to node [above]{$\bar{\partial}^{\infty}_{n}$} (B);
\draw[->] (D) to node [above]{$\bar{j}_{*}^{\infty}$}(C);
\draw[->] (E) to node [above]{$\bar{i}_{*}^{\infty}$}(D);
\draw[->] (F) to node [above]{}(E);
\end{tikzpicture}
\end{center}
and the \v{C}ech border cohomology sequence
\begin{center}
\begin{tikzpicture}
\node (A) {$\cdots$};
\node (B) [node distance=2cm, right of=A] {$\hat{H}^{n-1}_{\infty}(A,B;G)$};
\node (C) [node distance=3cm, right of=B] {$\hat{H}^{n}_{\infty}(X,A;G)$};
\node (D) [node distance=3cm, right of=C] {$\hat{H}^{n}_{\infty}(X,B;G)$};
\node (E) [node distance=3cm, right of=D] {$\hat{H}^{n}_{\infty}(A,B;G)$};
\node (F) [node distance=2cm, right of=E] {$\cdots$};
\draw[->] (A) to node [above]{}(B);
\draw[->] (B) to node [above]{$\bar{\delta}_{\infty}^{n}$} (C);
\draw[->] (C) to node [above]{$\bar{j}^{*}_{\infty}$}(D);
\draw[->] (D) to node [above]{$\bar{i}^{*}_{\infty}$}(E);
\draw[->] (E) to node [above]{}(F);
\end{tikzpicture}
\end{center}
are partially exact and exact, respectively. Here $\bar{\partial}_{n}^{\infty}=j_{n-1}^{'\infty}\cdot \partial_{n}^{\infty}$, $\bar{\delta}_{\infty}^{n}=\delta_{\infty}^{n}\cdot j_{\infty}^{' n-1}$ and $\bar{j}_{\infty}^{*}$ and $\bar{i}_{\infty}^{*}$ are the homomorphisms induced by the inclusion maps $j^{'}:A\to (A,B)$, $\bar{i}:(A,B)\to (X,B)$ and $\bar{j}:(X,B)\to (X,A)$.
\end{theorem}

\begin{proof}
The proof is similar to the proof of the corresponding Theorems 10.2 and 10.2c of [E-St] (see Ch. I,\textsection† 10).
\end{proof}

\section{On some applications of \v{C}ech border homology and cohomology groups}
Now we are mainly interested in the following problem: how to characterize the \v{C}ech homology and cohomology groups, coefficients of cyclicity, and cohomological dimensions of remainders of Stone-\v{C}ech compactifications of spaces.

Our main result about the connection between \v{C}ech (co)homology groups of remainders and \v{C}ech border (co)homology groups of spaces is.

\begin{theorem}
Let $(X,A)\in ob(\mathcal{M}^{2}_{p})$ and let $(\beta X,\beta A)$ be the pair of Stone-\v{C}ech compactifications of $X$ and $A$. Then 
\[\check{{H}}_{n}^{f}(\beta X\setminus X,\beta A\setminus A;G)=\check{{H}}_{n}^{\infty}(X,A;G)\] 
and
\[\hat{{H}}^{n}_{f}(\beta X\setminus X,\beta A\setminus A;G)=\hat{{H}}^{n}_{\infty}(X,A;G).\] 
\end{theorem}
\begin{proof}
Let $\alpha=\{\alpha_{v}\}_{v\in (V_{\alpha},V_{\alpha}^{\beta A\setminus A})}$ and $\alpha^{'}=\{\alpha^{'}_{w}\}_{w\in (W_{\alpha^{'}},W_{\alpha^{'}}^{\beta A\setminus A})}$ be the closed covers of pairs $(\beta X\setminus X,\beta A\setminus A)$ and $\alpha\geq \alpha^{'}$. By Lemma 4 of [Sm$_{4}$] there exist open swellings $\beta_{1}=\{\beta_{v}^{1}\}_{v\in (V_{\alpha},V_{\alpha}^{\beta A\setminus A})}$ and $\beta^{'}=\{\beta_{w}^{'}\}_{w\in (W_{\alpha^{'}},W_{\alpha^{'}}^{\beta A\setminus A})}$ of $\alpha$ and $\alpha^{'}$ in $\beta X$, respectively. Assume that $\alpha_{v}\subseteq \alpha^{'}_{w_{k}}$, $k=1,2,\cdots ,m_{v}$. Let
 \[\beta_{v}=\beta_{v}^{1}\cap (\bigcap\limits_{k=1}^{{{m}_{v}}}{\beta _{{{w}_{k}}}^{'}}),~~~~v\in V_{\alpha}.\]

Note that $\alpha_{v}\subset\beta_{v}\subset\beta_{v}^{1}$ for each $v\in V_{\alpha}$. It is clear that $\beta=\{\beta_{v}\}_{v\in (V_{\alpha},V_{\alpha}^{A})}$ is a swelling of $\alpha=\{\alpha_{v}\}_{v\in (V_{\alpha},V_{\alpha}^{\beta A\setminus A})}$ and $\beta\geq \beta^{'}$. 

The swelling in $\beta X$ of closed cover $\alpha$ of $(\beta X\setminus X,\beta A\setminus A)$ is denoted by $s(\alpha)$. Let $S$ be the set of all swellings of such kind.

Now define an order $\geq^{'}$ in $S$. By definition,
\[s(\alpha^{'})\geq^{'} s(\alpha)\Leftrightarrow s(\alpha^{'})\geq s(\alpha)~\wedge ~\alpha^{'}\geq \alpha.\]

It is clear that $S$ is directed by $\geq^{'}$. Let $((\beta X\setminus X)_{s(\alpha)},(\beta A\setminus A)_{s(\alpha)})$ be the nerve of $s(\alpha)\in S$ and $p_{s(\alpha)s(\alpha^{'})}$ be the projection simplicial map induced by the refinement $\alpha^{'}\geq\alpha$. Consider an inverse system
\[\{H_{n}((\beta X\setminus X)_{s(\alpha)},(\beta A\setminus A)_{s(\alpha)};G),p_{s(\alpha)*}^{s(\alpha^{'})},S\}\]
and a direct system
\[\{H^{n}((\beta X\setminus X)_{s(\alpha)},(\beta A\setminus A)_{s(\alpha)};G),p_{s(\alpha)}^{s(\alpha^{'})*},S\}.\]

Let $\varphi:S\to \rm{cov}_{f}^{cl}(\beta X\setminus X,\beta A\setminus A)$ be the function in the set of closed finite covers of pair $(\beta X\setminus X,\beta A\setminus A)$ given by formula
\[\varphi(s(\alpha))=\alpha,~~~~s(\alpha)\in S.\]

Note that $\varphi$ is an increasing function and 
\[\varphi(S)=\rm{cov}_{f}^{cl}(\beta X\setminus X,\beta A\setminus A).\]

For each index $s(\alpha)\in S$, we have
\[H_{n}((\beta X\setminus X)_{s(\alpha)},(\beta A\setminus A)_{s(\alpha)};G)=H_{n}((\beta X\setminus X)_{\alpha},(\beta A\setminus A)_{\alpha};G)\]
and
\[H^{n}((\beta X\setminus X)_{s(\alpha)},(\beta A\setminus A)_{s(\alpha)};G)=H^{n}((\beta X\setminus X)_{\alpha},(\beta A\setminus A)_{\alpha};G).\]

It is known that for normal spaces the \v{C}ech (co)homology groups based on finite open covers and on finite closed covers are isomorphic. By Theorems 3.14 and 4.13 of ([E-St],Ch.VIII) we have
\begin{equation}
\check{H}_{n}^{f}(\beta X\setminus X,\beta A\setminus A;G)\approx \underset{\longleftarrow}{\mathop{\lim }}\,\{{{H}_{n}}{{((\beta X\backslash X)_{s(\alpha )},(\beta A\backslash A)_{s(\alpha )}}};G),{{p}_{s(\alpha )*}^{s({{\alpha }^{'}})}},S\}
\end{equation}
and
\begin{equation}
\hat{H}^{n}_{f}(\beta X\setminus X,\beta A\setminus A;G)\approx \underset{\longrightarrow}{\mathop{\lim }}\,\{{{H}^{n}}({(\beta X\backslash X)_{s(\alpha )},(\beta A\backslash A)_{s(\alpha )}};G),{{p}_{s(\alpha )}^{s({{\alpha }^{'}})*}},S\}
\end{equation}
For each swelling $s(\alpha)=\{s(\alpha)_{v}\}_{v\in (V_{\alpha},V_{\alpha}^{\beta A\setminus A})}\in S$, the family 
\[s(\alpha)\wedge X=\{s(\alpha)_{v}\cap X\}_{v\in (V_{\alpha},V_{\alpha}^{\beta A\setminus A})}\]
is a border cover of $(X,A)$.

Let $\psi:S\to {\rm cov}_{\infty}(X,A)$ be the function defined by formula
\[\psi(s(\alpha))=s(\alpha)\wedge X,~~~~s(\alpha)\in S.\]

The function $\psi$ increases and $\psi(S)$ is a cofinal subset of ${\rm cov}_{\infty}(X,A)$. Note that the correspondence
\[((\beta X\setminus X)_{s(\alpha)},(\beta A\setminus A)_{s(\alpha)})\to (X_{s(\alpha)\wedge X},A_{s(\alpha)\wedge X}):s(\alpha)_{v}\to s(\alpha)_{v}\cap X,~~~v\in V_{\alpha}\]
induces an isomorphism of pairs of simplicial complexes. Thus, for each $s(\alpha)\in S$, we have the isomorphisms
\[H_{n}((\beta X\setminus X)_{s(\alpha)}, (\beta A\setminus A)_{s(\alpha)};G)=H_{n}(X_{s(\alpha)\wedge X},A_{s(\alpha)\wedge X};G)\]
and
\[H^{n}((\beta X\setminus X)_{s(\alpha)}, (\beta A\setminus A)_{s(\alpha)};G)=H^{n}(X_{s(\alpha)\wedge X},A_{s(\alpha)\wedge X};G).\]

By Theorems 3.15 and 4.13 of ([E-St],Ch.VIII), we have
\begin{equation}
\check{{H}}_{n}^{\infty}(X,A;G)=\underset{\longleftarrow}{\mathop{\lim }}\,\{{{H}_{n}}({(\beta X\backslash X)_{s(\alpha )},(\beta A\backslash A)_{s(\alpha )}};G),{{p}_{s(\alpha )*}^{s({{\alpha }^{'}})}},S\}
\end{equation}
and
\begin{equation}
\hat{{H}}_{\infty}^{n}(X,A;G)=\underset{\longrightarrow}{\mathop{\lim }}\,\{{{H}_{n}}({{(\beta X\backslash X,\beta A\backslash A)}_{s(\alpha )}};G),{{p}_{s(\alpha )}^{s({{\alpha }^{'}})*}},S\}.
\end{equation}

From (1), (2), (3) and (4) it follows that 
\[\check{{H}}_{n}^{\infty}(X,A;G)=\check{{H}}_{n}^{f}(\beta X\setminus X,\beta A\setminus A;G)\]
and 
\[\hat{{H}}^{n}_{\infty}(X,A;G)=\hat{{H}}^{n}_{f}(\beta X\setminus X,\beta A\setminus A;G).\]
\end{proof}

The cohomological coefficient of cyclicity $\eta_{G}(X,A)$ of pair $(X,A)$ was defined by S.Novak [N] and M.F.Bokstein [Bo]. Dually one can define the homological coefficient of cyclicity $\eta^{G}(X,A)$ of pair $(X,A)$.

Now give the following definitions and results.
\begin{Definition}
Let $G$ be an abelian group and $n$ nonnegative integer. A border (co)homological coefficient of cyclicity of pair $(X,A)\in ob (\mathcal{M}^{2}_{p})$ with respect to $G$ denoted by $(\eta_{G}^{\infty}(X,A))$ $\eta_{\infty}^{G}(X,A)$ is $n$, if $(\hat{H}_{\infty}^{m}(X,A;G)=0)$ $\check{H}_{m}^{\infty}(X,A;G)=0$ for all $m>n$ and $(\hat{H}_{\infty}^{n}(X,A;G)\neq 0)$ $\check{H}^{\infty}_{n}(X,A;G)\neq 0$.

Finally, $(\eta_{G}^{\infty}(X,A)=+\infty )$ $\eta^{G}_{\infty}(X,A)=+\infty$ if for every $m$ there is $n\geq m$ with $(\hat{H}_{\infty}^{n}(X,A;G)\neq 0)$ $\check{H}^{\infty}_{n}(X,A;G)\neq 0$.
\end{Definition}

\begin{theorem}
For each pair $(X,A)\in ob(\mathcal{M}^{2}_{p})$,
\[\eta_{G}^{\infty}(X,A)=\eta_{G}(\beta X\setminus X,\beta A\setminus A)\]
and
\[\eta^{G}_{\infty}(X,A)=\eta^{G}(\beta X\setminus X,\beta A\setminus A).\]
\end{theorem}
\begin{proof}
This is an immediate consequence of Theorem 2.1. Indeed, let $\eta_{G}(\beta X\setminus X,\beta A\setminus A)=n$. Then for each $m>n$, $\hat{H}_{f}^{m}(\beta X\setminus X,\beta A\setminus A;G)=0$ and $\hat{H}_{f}^{n}(\beta X\setminus X,\beta A\setminus A;G)\neq 0$. From the isomorphism
\[\hat{H}_{f}^{k}(\beta X\setminus X,\beta A\setminus A;G)=\hat{H}_{f}^{k}(X,A;G)\]
it follows that $\hat{H}_{\infty}^{m}(X,A;G)=0$ for each $m>n$, and $\hat{H}_{\infty}^{n}(X,A;G)\neq 0$. Thus, $\eta_{G}^{\infty}(X,A)=n=\eta_{G}(\beta X\setminus X, \beta A\setminus A)$.

Analogously we can prove equality $\eta^{G}_{\infty}(X,A)=\eta^{G}(\beta X\setminus X,\beta A\setminus A).$
\end{proof}

The theory of cohomological dimension has become an important branch of dimension theory since A. Dranishnikov solved P.S. Alexandrov's problem and developed the theory of extension dimension ([D], [D-Dy]).

Our next aim is to study some questions of theory of cohomological dimension. In particular, we now give a description of cohomological dimension of remainder of Stone-\v{C}ech compactification of metrizable space.

Following Y. Kodama (see the appendix of [N]) and T. Miyata [Mi] we give the following definition.

\begin{Definition}
The border small cohomological dimension $d_{\infty}^{f}(X;G)$ of normal space $X$ with respect to group $G$ is defined to be the smallest integer $n$ such that, whenever $m\geq n$ and $A$ is closed in $X$, the homomorphism $i^{*}_{A,\infty}:\hat{H}_{\infty}^{m}(X;G)\to \hat{\rm{H}}_{\infty}^{m}(A;G)$ induced by the inclusion $i:A\to X$ is an epimorphism. 

The border small cohomological dimension of $X$ with coefficient group $G$ is a function $d^{f}_{\infty}:\mathcal{N}\to {\rm N}\cup \{0,+ \infty\}:X\to n$, where $d_{\infty}^{f}(X;G)=n$ and ${\rm N}$ is the set of all positive integers.
\end{Definition}
\begin{theorem}
Let $X$ be a metrizable space. Then the following equality 
\[d^{f}_{\infty}(X;G)= d_{f}(\beta X\setminus X;G)\]
holds, where $d_{f}(\beta X\setminus X;G)$ is the small cohomological dimension of $\beta X\setminus X$ (see $\rm{[N], p.199}$).
\end{theorem}
\begin{proof}
Let $A$ be a closed subset of $X$. Assume that $d_{f}(\beta X\setminus X;G)=n$. Then for each $m\geq n$ the homomorphism $i^{*}_{\beta X\setminus X,\infty}:\hat{H}_{f}^{m}(\beta X\setminus X;G)\to \hat{H}_{f}^{m}(\beta A\setminus A;G)$ is an epimorphim. Consider the following commutative diagram 

\begin{center}
\begin{tikzpicture}
\node (A) {$\hat{\rm{H}}^{m}_{\infty}(X;G)$};
\node (B) [node distance=4cm, right of=A] {$\hat{\rm{H}}^{m}_{f}(\beta X\setminus X;G)$};
\node (B1) [node distance=1cm, below of=A] {};
\node (B2) [node distance=7cm, right of=B1] {$(5)$};
\node (C) [node distance=2cm, below of=A] {$\hat{\rm{H}}^{m}_{\infty}(A;G)$};
\node (D) [node distance=4cm, right of=C] {$\hat{\rm{H}}^{m}_{f}(\beta A\setminus A;G).$};
\draw[->] (A) to node [left]{$i_{A,\infty}^{*}$} (C);
\draw[->] (B) to node [right]{$i_{\beta A\setminus A}^{*}$}(D);
\draw[transparent] (A) edge node[rotate=0,opacity=1] {$\approx$} (B);
\draw[transparent] (C) edge node[rotate=0,opacity=1] {$\approx$} (D);

\end{tikzpicture}
\end{center}

It is clear that the homomorphim
\[i^{*}_{A,\infty}:\hat{H}_{f}^{m}(X;G)\to \hat{H}_{f}^{m}(A;G)\]
also is an epimorphim for each $m\geq n$. Thus, 
$$d^{f}_{\infty}(X;G)\leq n=d_{f}(\beta X\setminus X;G).~~~~~~~~~~~~(6)$$

Let $d_{\infty}^{f}(X;G)=n$. To see the reverse inequality, let $B$ be a closed subset of $\beta X\setminus X$ and let $m\geq n$.

Consider an open in $\beta X\setminus X$ neighbourhood $U$ of $B$. There exists an open neighbourhood $V$ of $B$ in $\beta X\setminus X$ such that $\bar{V}^{\beta X\setminus X}\subset U$. By Lemma 5 of [Sm$_{4}$] we can find an open set $W$ in $\beta X$ such that $W\cap(\beta X\setminus X)=V$ and $\bar{W}^{\beta X}\cap (\beta X\setminus X)\subseteq U$. Let $A=\bar{W}^{\beta X}\cap X$. It is clear that $\beta A=\bar{A}^{\beta X}$.

We have 
\[\bar{W}^{\beta X}=\overline{W\cap X}^{\beta X}\subset \overline{\bar{W}^{\beta X}\cap X}^{\beta X}\subset \overline{\bar{W}^{\beta X}}^{\beta X}=\bar{W}^{\beta X}.\]

Consequently, $\beta A=\overline{\bar{W}^{\beta X}\cap X}^{\beta X}=\bar{W}^{\beta X}$. This shows that 
\[B\subset \beta A\cap (\beta X\setminus X)\subset U. \]

Hence, we have 
\[B\subset \beta A\setminus A\subset U.\]

Thus, for each closed set $B$ of $\beta X\setminus X$ and its open neighbourhood $U$ in $\beta X\setminus X$ there exists a closed subset $A$ in $X$ such that $B\subset \beta A\setminus A\subset U$.

Let $a\in H^{n}_{f}(B;G)$. There is a closed finite cover $\alpha$ of $B$ such that an element $a_{\alpha}\in H^{m}(N(\alpha);G)$ represents the element $a$. 

Using Lemma 4 of [Sm$_{4}$] we can find the swellings $\tilde{\alpha}$ and $\tilde{\tilde{\alpha}}$ of $\alpha$ in $B$ and $\beta X\setminus X$, respectively, such that $\tilde{\tilde{\alpha}}_{|B}=\tilde{\alpha}$. Let $U$ be the union of elements of $\tilde{\tilde{\alpha}}$. There is a closed set $A$ of $X$ with $B\subset \beta A\setminus A\subset U$. The nerves $N(\alpha)$, $N(\tilde{\alpha})$ and $N(\tilde{\tilde{\alpha}}_{|\beta A \setminus A})$ are isomorphic. We can assume that 
$$H^{n}(N(\alpha);G)=H^{n}(N(\tilde{\alpha});G)=H^{n}(N(\tilde{\tilde{\alpha}}_{|\beta A\setminus A});G).$$
Hence, the element $a_{\alpha}$ also belongs to the group $H^{n}(N(\tilde{\tilde{\alpha}}_{|\beta A\setminus A});G)$. Consequently, it represents some element $b$ of $\hat{H}^{n}(\beta A\setminus A;G)$.

The inclusion $i_{A}:A\to X$ induces an epimorphism $i^{*}_{A,\infty}:\hat{H}^{m}_{\infty}(X;G)\to \hat{H}^{m}(A;G)$. From diagram $(5)$ it follows that the homomorphism $i^{*}_{\beta A\setminus A}:\hat{H}^{m}(\beta X\setminus X;G)\to \hat{H}^{m}(\beta A\setminus A;G)$ is an epimorphism. Consequently, there is an element $c\in \hat{H}^{m}(\beta X\setminus X;G)$ such that $i^{*}_{\beta A\setminus A}(c)=b$. The homomorphism $j_{B}^{*}:\hat{H}^{m}(\beta A\setminus A;G)\to \check{H}^{m}(B;G)$ induced by the inclusion $j_{B}:B\to \beta A\setminus A$ satisfies the condition $j^{*}_{B}(b)=a$. From equality $i_{\beta A\setminus A}\cdot j_{B}=i_{B}$ it follows that $i^{*}_{B}(c)=a$.

Thus the inclusion $i_{B}:B\to \beta X\setminus X$ also induces an epimorphism $i_{B}^{*}:\check{H}^{m}(\beta X\setminus X;G)\to \check{H}^{m}(B;G)$. Hence, we obtaine
\[d_{f}(\beta X\setminus X;G)\leq n=d_{\infty}^{f}(X;G).~~~~~~~~~(7)\]

From the inequalities $(6)$ and $(7)$ it follows that 
\[d_{\infty}^{f}(X;G)=d_{f}(\beta X\setminus X;G).\] 
\end{proof}
\begin{theorem}
Let $A$ be a closed subspace of a normal space $X$. Then 
\[d^{f}_{\infty}(A;G)\leq d^{f}_{\infty}(X;G).\] 
\end{theorem}
\begin{proof}
Let $B$ be an arbitrary closed subset of $A$ and $j_{B}:B\to A$, $i_{A}:A\to X$ and $k_{B}:B\to X$ be the inclusion maps. Note that $k_{B}=i_{A}\cdot j_{B}$. The induced homomorphisms $k_{B,\infty}^{*}:\hat{H}_{\infty}^{n}(X;G)\to \hat{H}_{\infty}^{n}(B;G)$, $i_{A,\infty}^{*}:\hat{H}_{\infty}^{n}(X;G)\to \hat{H}_{\infty}^{n}(A;G)$ and $j_{B,\infty}^{*}:\hat{H}_{\infty}^{n}(A;G)\to \hat{H}_{\infty}^{n}(B;G)$ satisfy the equality $k_{B,\infty}^{*}=j_{B,\infty}^{*}\cdot i_{A,\infty}^{*}$.

Let $n=d_{f}^{\infty}(X;G)$. For each $m\geq n$, the homomorphisms $k^{*}_{B,\infty}:\hat{H}_{\infty}^{m}(X;G)\to \hat{H}_{\infty}^{m}(B;G)$ and $i_{A,\infty}^{*}:\hat{H}_{\infty}^{m}(X;G)\to \hat{H}_{\infty}^{m}(A;G)$ are epimorphisms. Hence, the homomorphism $j_{B,\infty}^{*}:\hat{H}_{\infty}^{m}(A;G)\to \hat{H}_{\infty}^{m}(B;G)$ is also an ephimorphism for each $m\geq n$. Thus, $d_{f}^{\infty}(A;G)\leq n=d_{f}^{\infty}(X;G)$.
\end{proof}
\begin{corollary}
For each closed subspace $A$ of a metrizable space $X$,
\[d^{f}_{\infty}(A;G)\leq d_{f}(\beta X\setminus X;G).\]
\end{corollary}
\begin{Definition}
The border large cohomological dimension $D_{\infty}^{f}(X;G)$ of normal space $X$ with respect to group $G$ is defined to be the largest integer $n$ such that $\hat{H}_{\infty}^{n}(X,A;G)\neq 0$ for some closed set $A$ of $X$. 

The border large cohomological dimension of $X$ with coefficient group $G$ is a function $D_{\infty}^{f}:\mathcal{N}\to {\rm N}\cup \{0,+ \infty\}:X\to n$, where $D_{\infty}^{f}(X;G)=n$ and ${\rm N}$ is the set of all positive integers.
\end{Definition}
\begin{theorem}
For each metrizable space $X$, one has
\[D_{\infty}^{f}(X;G)= D_{f}(\beta X\setminus X;G),\]
where $D_{f}(\beta X\setminus X;G)$ is the large cohomological dimension of $\beta X\setminus X$ (see $\rm{[N], p.199}$).
\end{theorem}
\begin{proof}
Let $D_{f}(\beta X\setminus X;G)=n$. Consider an arbitrary closed subspace $A$ of $X$. The remainder $\beta A\setminus A$ is a closed subset of $\beta X\setminus X$. By the assumption, we have $\hat{H}^{m}(\beta X\setminus X, \beta A\setminus A;G)=0$ for each $m>n$. Theorem 2.1 implies that $\hat{H}^{m}_{\infty}(X,A;G)=0$ for each $m>n$ and $A\subset X$. Thus,
\[D_{\infty}^{f}(X;G)\leq n=D_{f}(\beta X\setminus X;G).~~~~~~~~~~~~(8)\]

Let $D^{f}_{\infty}(X;G)=n$. Assume that $D_{f}(\beta X\setminus X;G)=n_{1} > n$. Then there is a closed set $B$ in $\beta X\setminus X$  such that $\hat{H}^{n_{1}}(\beta X\setminus X, B;G)\neq 0$. Using Lemma 4 of [Sm$_{4}$] and the proof of Theorem 2.5 we can show that there is a closed set $A$ of $X$ such that $B\subset \beta A\setminus A$, and $\hat{H}^{n_{1}}(\beta X\setminus X, \beta A\setminus A;G)\neq 0$. By Theorem 2.1 $\hat{H}^{n_{1}}_{\infty}(X,A;G)\neq 0$. But it is not possible because $D_{f}^{\infty}(X;G)=n$. Therefore, $n_{1}\leq n$. Thus,
\[D_{f}(\beta X\setminus X;G)\leq n= D^{f}_{\infty}(X;G)~~~~~~~~~~~~(9)\]
The inequalities (8) and (9) imply
\[D^{f}_{\infty}(X;G)=D_{f}(\beta X\setminus X;G).\]
\end{proof}
\begin{theorem}
If $A$ is a closed subset of normal space $X$, then 
\[D_{\infty}^{f}(A;G)\leq D_{\infty}^{f}(X;G).\]
\end{theorem}
\begin{proof}
By Theorem 1.10, for each closed set $B$ of $A$, there is the exact \v{C}ech border cohomological sequence
\begin{center}
\begin{tikzpicture}
\node (A) {$\cdots$};
\node (B) [node distance=2cm, right of=A] {$\hat{H}^{m-1}_{\infty}(A;G)$};
\node (C) [node distance=3cm, right of=B] {$\hat{H}^{m}_{\infty}(X,A;G)$};
\node (D) [node distance=3cm, right of=C] {$\hat{H}^{m}_{\infty}(X,B;G)$};
\node (E) [node distance=3cm, right of=D] {$\hat{H}^{m}_{\infty}(A,B;G)$};
\node (F) [node distance=2cm, right of=E] {$\cdots$};
\draw[->] (A) to node [above]{}(B);
\draw[->] (B) to node [above]{$\bar{\delta}_{\infty}^{m}$} (C);
\draw[->] (C) to node [above]{$\bar{j}^{*}_{\infty}$}(D);
\draw[->] (D) to node [above]{$\bar{i}^{*}_{\infty}$}(E);
\draw[->] (E) to node [above]{}(F);
\end{tikzpicture}
\end{center}

It is clear that, if $m>D_{\infty}^{f}(X;G)$, then $\hat{H}_{\infty}^{m}(X,A;G)=\hat{H}_{\infty}^{m}(X,B;G)=0$. Consequently, $\hat{H}^{m}_{\infty}(A,B;G)=0$. Thus, we have
\[D_{\infty}^{f}(A;G)\leq D_{\infty}^{f}(X;G).\]
\end{proof}
\begin{corollary}
For each closed subspace $A$ of metrizable space $X$, one has
\[D^{f}_{\infty}(A;G)\leq D_{f}(\beta X\setminus X;G).\]
\end{corollary}
\begin{theorem}
If $X$ is a normal space then 
\[d_{\infty}^{f}(X;G)\leq D_{\infty}^{f}(X;G).\]
\end{theorem}
\begin{proof}
Let $A$ be a closed subset of normal space $X$. Consider the exact \v{C}ech border cohomological sequence of pair $(X,A)$
\begin{center}
\begin{tikzpicture}
\node (A) {$\cdots$};
\node (B) [node distance=2cm, right of=A] {$\hat{H}^{m-1}_{\infty}(A,B;G)$};
\node (C) [node distance=3cm, right of=B] {$\hat{H}^{m}_{\infty}(X,A;G)$};
\node (D) [node distance=3cm, right of=C] {$\hat{H}^{m}_{\infty}(X;G)$};
\node (E) [node distance=3cm, right of=D] {$\hat{H}^{m}_{\infty}(A;G)$};
\node (F) [node distance=2cm, right of=E] {$\cdots$};
\draw[->] (A) to node [above]{}(B);
\draw[->] (B) to node [above]{${\delta}_{\infty}^{m}$} (C);
\draw[->] (C) to node [above]{${j}^{*}_{\infty}$}(D);
\draw[->] (D) to node [above]{${i}^{*}_{\infty}$}(E);
\draw[->] (E) to node [above]{}(F);
\end{tikzpicture}
\end{center}

Let $m>D_{\infty}^{f}(X;G)$. Note that $j_{\infty}^{*}:\hat{H}_{\infty}^{m-1}(X;G)\to \hat{H}_{\infty}^{m-1}(A;G)$ is an epimorphism. Hence, 
\[d_{\infty}^{f}(X;G)\leq D_{\infty}^{f}(X;G).\]
\end{proof}
\begin{corollary}
For each metrizable space $X$, one has
\[d_{f}(\beta X\setminus X;G)\leq D^{f}_{\infty}(X;G)\]
and 
\[d^{f}_{\infty}(X;G)\leq D_{f}(\beta X\setminus X;G).\]
\end{corollary}
\begin{remark}
The results of this paper also hold for spaces satisfying the compact axiom of countability. Recall that a space $X$ satisfies the compact axiom of countability if for each compact subset $B\subset X$ there exists a compact subset $B^{'}\subset X$ such that $B\subset B^{'}$ and $B^{'}$ has a countable or finite fundamental systems of neighbourhoods (see Definition 4 of [Sm$_{4}$], p.143). A space $X$ is complete in the sense of \v{C}ech if and only if it is $G_{\delta}$ type set in some compact extension. Each locally metrizable spaces, complete in the seance of \v{C}ech spaces [\v{C}] and locally compact spaces satisfy the compact axiom of countability.
\end{remark}

\end{document}